\documentclass[a4paper, 12pt, oneside]{amsart}

\usepackage{amsfonts}
\usepackage{amstext}
\usepackage{amsbsy}
\usepackage{amsopn}
\usepackage{eucal}
\usepackage[initials, non-sorted-cites]{amsrefs}
\usepackage{vmargin}
\usepackage[pdftex]{graphicx}
\usepackage{mathtools}
\usepackage{amssymb}
\usepackage{mathrsfs}
\usepackage{stmaryrd}
\usepackage{amsmath}
\usepackage{amsthm}
\usepackage{comment}
\usepackage{enumerate}
\usepackage[usenames,dvipsnames]{xcolor}

\usepackage{times}

\usepackage[T1]{fontenc}

\usepackage[active]{srcltx}

\usepackage{enumitem}

\newcommand{\bbR}{\mathbb{R}}
\newcommand{\bbS}{\mathbb{S}}

\newcommand{\bbN}{\mathbb{N}}
\newcommand{\mfg}{\mathfrak{g}} 

\newcommand{\mfk}{\mathfrak{k}}

\newcommand{\mfm}{\mathfrak{m}}

\newcommand{\tX}{\tilde{X}}

\newcommand{\Ad}{\mathrm{Ad}}
\newcommand{\GL}{\mathrm{GL}}
\newcommand{\grad}{\mathrm{grad}}
\newcommand{\tr}{\mathrm{tr}}

\newcommand{\remark}{\noindent\textbf{Remark}}

\newtheorem{theorem}{Theorem}[section]

\newtheorem{lemma}[theorem]{Lemma}
\newtheorem{proposition}[theorem]{Proposition}

\author{Gabor Lippner \textsuperscript{$1$}}
\author{Dan Mangoubi \textsuperscript{$2$}}
\author{Zachary McGuirk \textsuperscript{$3$}}
\author{Rachel Yovel \textsuperscript{$4$}}

\address{\textsuperscript{$1$} \small Gabor Lippner, Northeastern University, 360 Huntington Ave, Boston, MA 02115, United States} 
\email{g.lippner@northeastern.edu}
\address{\textsuperscript{$2$} \small Dan Mangoubi, Einstein Institute of Mathematics, Edmond J. Safra Campus, The Hebrew University of Jerusalem, Jerusalem 91904, Israel} 
\email{dan.mangoubi@mail.huji.ac.il}
\address{\textsuperscript{$3$} \small Zachary McGuirk, Einstein Institute of Mathematics, Edmond J. Safra Campus, The Hebrew University of Jerusalem, Jerusalem 91904, Israel} 
\email{zachary.mcguirk@mail.huji.ac.il}
\address{\textsuperscript{$4$} \small Rachel Yovel, Einstein Institute of Mathematics, Edmond J. Safra Campus, The Hebrew University of Jerusalem, Jerusalem 91904, Israel} 
\email{rachel.yovel@mail.huji.ac.il}


\begin{document}

\title{Strong Convexity for harmonic functions on compact symmetric spaces}

\begin{abstract}
Let $h$ be a harmonic function defined on a spherical disk. It is shown that~$\Delta^k |h|^2$ is nonnegative for all~$k\in\bbN$ where $\Delta$ is the Laplace-Beltrami operator. This fact is generalized
to harmonic functions defined on a disk in a  normal homogeneous compact Riemannian manifold, and in particular in a symmetric space of the compact type.
This complements a similar property for harmonic functions on $\bbR^n$ discovered by the first two authors
and is related to strong convexity of the $L^2$-growth function of harmonic functions.
\end{abstract}

\maketitle

\section{Introduction}
Let $\Delta_E$ denote the Laplace operator on the Eudlidean $n$-dimensional space, $\bbR^n$, and let $h$ be a harmonic function defined on a disk in~$\bbR^n$. It was proved in \cite{lipp-man} that $\Delta_E^k |h|^2$ is nonnegative for all~$k$. 
The main motivation to study powers of the Laplace operators comes from the well known classical fact that the Taylor series expansion
of the spherical mean function is expressed in terms of the powers $\Delta^k$ (see e.g.~\cite{poritsky}), and as such 
this positivity statement has given a new argument for the logarithmic convexity of the~$L^2$-growth function of harmonic functions \cite{agmon} or, equivalently, the monotonicity of the frequency function~\cite{almgren}. Moreover, this series of linear inequalities was important in deducing a \emph{discrete} version of the three circles theorems for harmonic functions.

The main goal of this paper is to show that, maybe surprisingly, an analogous convexity phenomenon also holds on the standard sphere $\bbS^n\subset \bbR^{n+1}$ and more generally on symmetric spaces of the compact type.
In fact, it is shown to hold on a  more general family of spaces, namely, normal homogeneous compact Riemannian manifolds, which will be discussed in~\S\ref{sec:normal-homogeneous}.
 
To avoid an overly technical introduction, we first state the theorem for spheres.
Let~$\Delta_S$ denote the Laplace-Beltrami operator on the standard unit sphere, $\bbS^n$, and denote by $B\left(\rho\right)  \subset \bbS^{n}$  a ball of radius $0<\rho<\pi$.  We prove
\begin{theorem} 
\label{thm:spheres}
Let $h\colon B\left(\rho\right) \rightarrow \mathbb{R}$ be a harmonic function. Then, 
$$\forall k\in\bbN\quad \Delta_S^k |h|^2 \geq 0 $$ 
\end{theorem}

More generally, we show
\begin{theorem}
\label{thm:homog}
Let $G$ be a semisimple, connected, compact Lie group with a fixed bi-invariant Riemannian metric. 
Let $K$ be a closed subgroup. Consider the homogeneous Riemannian manifold $M=G/K$ with the corresponding induced metric. Let~$h$ be a harmonic function defined on an open subset of~$M$. 
Then $$\forall k\in\bbN\quad  \Delta_M^k |h|^2\geq 0$$
where $\Delta_M$ is the Laplace-Beltrami operator on $M$.
\end{theorem}

 On compact Riemannian symmetric spaces of rank one or on two-point compact homogeneous spaces one can interpret Theorem~\ref{thm:homog}
 as an absolute monotonicity result for the spherical mean function,
 since in these cases the spherical mean function can be expressed in terms of the Laplace powers~\cite{helgason-acta}.

The proof in \cite{lipp-man} is strongly based on the invariance of harmonic functions under translations. On a general Lie group equipped with an invariant Riemannian structure, harmonic functions are invariant under right translations, while, clearly, in the class of Lie groups possessing bi-invariant Riemannian structures we again recover the invariance of harmonic functions under all left and right translations.
 The paper deals with harmonic functions on quotient spaces of Lie groups equipped with a bi-invariant Riemannian metric.
The proof of Theorems~\ref{thm:spheres} and~\ref{thm:homog} is based on the possibility to write the Laplace-Beltrami operator~$\Delta$ as a sum of squares 
of vector fields which commute with $\Delta$. While in~$\bbR^n$ one can represent the Laplace-Beltrami operator as a sum of~$n$ squares, one needs more than~$n=\dim M$ vector fields to do it on $M=G/K$.
The number of vector fields needed is the dimension of~$G$.
\vspace{2ex}

\noindent\textbf{Organization of the paper.}
Although Theorem~\ref{thm:homog} generalizes the special case in Theorem~\ref{thm:spheres}, we give the proof of that special case first as it is  direct and different than the one for the general case. In addition, the proof in the case of the sphere can be made without any reference to notions from Lie Theory.\vspace{1ex}

\noindent\textbf{Acknowledgements}. Theorem \ref{thm:spheres} was proved in \cite{yovel} by long computations. We are grateful to Nir Avni for a valuable suggestion which led us to find a conceptual explanation for the result in \cite{yovel}. The correct class of homogeneous spaces for which our result in the spherical setting holds was unclear to us for a while and we are grateful to Joseph Bernstein for directing us towards the class of symmetric spaces. R.Y. would like to thank Pavel Giterman for helpful insights.
D.M., Z.M. and R.Y. were supported by ISF grant nos. 753/14 and 681/18.
G.L. and D.M. were supported by BSF grant no. 2018174.

\section{The case of the sphere: Proof of Theorem \ref{thm:spheres}}
For the $n$-dimensional Euclidean space, the Laplace operator can be written as a sum of squares of $n$ orthogonal constant vector fields. The proposition below shows that for the unit sphere~$\bbS^n$, the Laplace-Beltrami operator $\Delta_S$ can be written as a sum of squares of~$\binom{n+1}{2}$ rotation generating vector fields. This fact is well known and important in quantum mechanics of the hydrogen atom (see the two dimensional case in \cite{kirillov}*{ch. 4.9}). We include its proof for completeness and as an alternative to the general Lie group theory proof in the next section.

First, we define the vector fields concerned. Let~$x_1,\dots, x_{n+1}$ be the standard coordinates for~$\mathbb{R}^{n+1}.$ Let $R_{ij}(t)$ be the rotation of angle $t$ in the oriented $x_ix_j$ plane. Then, we define $X_{ij}$ to be the vector field generating the flow $R_{ij}(t),$ namely, $X_{ij}=x_i\partial_j-x_j\partial_i.$ 

\begin{proposition} \label{prop:spheres-sum-of-squares}
The Laplace-Beltrami operator on the sphere can be written as
\[\Delta_S=\sum _{1\leq i<j\leq n+1} X_{ij}^2.\]
\end{proposition}

\begin{proof}
Let $f:\bbS^{n} \rightarrow \mathbb{R}$ be a function, and let $\tilde{f}\colon\mathbb{R}^{n+1}\setminus{\{0\}}\rightarrow\mathbb{R}$ be its dilation invariant extension to $\bbR^{n+1}\setminus{\{0\}}$. Then,
\begin{align}\label{Xij^2}
\sum_{i<j} X_{ij}^2 \tilde{f} & = \sum_{i<j} \left( x_i^2\partial_j^2 \tilde{f} + x_j^2\partial_i^2 \tilde{f} - 2x_i x_j \partial_i\partial_j \tilde{f}+x_i\partial_i\tilde{f}+x_j\partial_j\tilde{f} \right) \\
& = \sum_{i\neq j} \left( x_i^2 \partial_j^2 \tilde{f} - x_i x_j \partial_i\partial_j \tilde{f} \right)-n\sum_{i} x_i\partial_i\tilde{f}.\nonumber
\end{align}
Since $\tilde{f}$ is constant along the radial direction, 
\begin{equation}\label{constant-r-direction}
\sum_{i=1}^{n+1} x_i \partial_i \tilde{f} = r\partial_r\tilde{f}=0.
\end{equation}
Fixing $j$ and using \eqref{constant-r-direction}, we have
\begin{equation}\label{fixedj}
\sum_{\substack{i=1,\\ i\neq j}}^{n+1} x_i x_j \partial_i \partial_j \tilde{f} = x_j \partial_j \left(\sum_{\substack{i=1,\\ i\neq j}}^{n+1} x_i \partial_i \tilde{f} \right) = -\left(x_j\partial_j \right)^2 \tilde{f}.
\end{equation}
Summing \eqref{fixedj} over $1\leq j \leq n+1$  we obtain
\begin{equation}\label{sumj}
\sum_{i\neq j}x_i x_j \partial_i \partial_j  \tilde{f} = -\sum_{j=1}^{n+1}x_j^2\partial_j^2\tilde{f}.
\end{equation}
From~\eqref{Xij^2},~\eqref{constant-r-direction} and~\eqref{sumj} we get
$$\sum_{i<j} X_{ij}^2 \tilde{f}  = \sum_{i\neq j} x_i^2 \partial_j^2 \tilde{f} +\sum_{j=1}^{n+1} x_j^2 \partial_j^2 \tilde{f} = r^2\sum_{j=1}^{n+1} \partial_j^2 \tilde{f}=r^2 \Delta_E \tilde{f}. $$
It only remains to recall the formula $\Delta_E=\partial_r^2+\frac{n-1}{r}\partial_r+\frac{1}{r^2}\Delta_S,$ to get
\[
\sum_{i<j} X_{ij}^2 f = \Delta_S f.
\]
\end{proof}

\begin{proof} [Proof of Theorem \ref{thm:spheres}]
Let $h: B\left(\rho\right)\rightarrow\mathbb{R}$ be a harmonic function. Using Proposition~\ref{prop:spheres-sum-of-squares} we notice that
$$
\Delta_S |h|^2 = 2\Re(\bar{h}\Delta_S h) +  2\sum_{i<j}\left|X_{ij}h\right|^2 = 2\sum_{i<j}\left|X_{ij}h\right|^2. 
$$
Since $X_{ij}$ commutes with the spherical Laplacian this formula expresses $\Delta_S |h|^2$ as a sum of squares of harmonic functions. Therefore, by induction, we conclude that $\Delta_S^k |h|^2\geq 0$ for every $k\in\mathbb{N}.$
\end{proof}

\section{The Laplacian on normal homogeneous Riemannian manifolds}
\label{sec:normal-homogeneous}
\subsection{General Lie group notations}
Let $G$ be a semisimple, connected, compact Lie group. 
Let $K$ be a closed subgroup of $G$. 
Following Helgason's notation from \cite{helgason-book1}*{Ch. II},
for $X\in T_e G$ we denote by 
$\tX$ the unique left invariant vector field on $G$ such that
$\tX_e=X$. Similarly, $\bar{X}$ denotes the corresponding right invariant vector field. We consider $T_e G$ with the Lie algebra structure induced from the Lie brackets
of left invariant vector fields and denote it by~$\mfg$.
For every $X\in \mfg$ there exists a unique one parametric subgroup $\gamma_X(t)$
such that $\dot{\gamma}(0)=X$. This subgroup is denoted by $\exp tX$.

Let $\mfk$ be the Lie subalgebra of $\mfg$ corresponding to a closed subgroup~$K$.
Let $L_g,$ and $R_g$ denote the left and right translations, respectively, in $G$ by an element~$g$. Then 
$$\Ad:G\to \GL(\mfg)$$ is the adjoint representation of $G$,
where $$\Ad(g)=(d R_{g^{-1}})|_g (d L_{g})|_e.$$

The homogeneous manifold $M=G/K$ is the space of left $K$-cosets. The canonical projection map is $\pi: G\to M$.
The coset $K$ will be called the ``origin'' of $M$, and will be denoted by $o$.
The left action of $g\in G$ on $M$ is denoted by $\tau_g$. When convenient we write $go=\tau_g\cdot o=\pi(g)$.

  
Given $X\in\mfg$, we denote by $X^+$ \cite{helgason-book1}*{Ch. II, \S3} the vector field on~$M$ generating the flow $p\mapsto \tau_{\exp tX}\cdot p$, i.e.
$$X^{+}f(p):=\frac{d}{dt}\Big|_{t=0} f(\tau_{\exp tX}\cdot p)\ .$$
It is useful to note the following identities:
\begin{eqnarray}
  d\pi|_g (\bar{X})&=&X^+_{go}\ , \\
   d\tau_g|_{ho}\circ d\pi|_h  &=& d\pi|_{gh}\circ d L_g|_h\ ,\\
    \label{identity:tau-Ad-correspondence}
     d\tau_g (X^+) &=&(\Ad(g)X)^+\ ,\\
     \label{form:X+Y+commutator}
  [X^{+}, Y^{+}]&=&-[X, Y]^{+}\ .
\end{eqnarray}
For~\eqref{form:X+Y+commutator} see \cite{helgason-book1}*{Ch. II, Theorem 3.4}.

\subsection{A bi-invariant Riemannian metric on~$G$ and an induced Riemannian metric on~$G/K$}
Fix a positive definite $\Ad(G)$-invariant bilinear form on~$\mfg$,
$$B:\mfg\times\mfg\to\bbR$$
such that:
\begin{equation}
\label{cdn:ad-invariance}
    \forall X, Y, Z\in\mfg,\quad B([Z, X], Y)+B(X, [Z, Y])=0\ ,
\end{equation}
where the latter identity is equivalent (as $G$ is connected) to
$$ \forall g\in G\quad B(\Ad(g) X, \Ad(g) Y)=B(X, Y)\ . $$
By our assumptions, positive definite $\Ad(G)$-invariant forms on~$\mfg$ exist
(\cite{helgason-book1}*{Ch. II, Prop 6.6}).
The form~$B$ on~$\mfg$ induces a bi-invariant Riemannian metric on~$G$,
$$\langle V, W \rangle_g := B\left((d L_g)^{-1} V, (d L_g)^{-1} W\right).$$

To define a $G$-invariant Riemannian metric on $M$, let $\mfm=\mfk^{\perp}$ be the orthogonal complement of $\mfk$ in~$\mfg$ with respect
to the quadratic form~$B$.
Since $B$ is in particular $\mathrm{Ad(K)}$-invariant, the linear decomposition $\mfg=\mfk\oplus\mfm$
is likewise $\mathrm{Ad}(K)$-invariant. 
  Observe that $\ker(d \pi)|_e=\mfk$ and thus, $d \pi|_{e}:\mfm\to T_o M$ is an isomorphism.
 Every $X\in\mfg$ can therefore be uniquely decomposed as,
 $$X=X_{\mfm}+X_{\mfk},$$
 where $X_{\mfm}\in\mfm$ and $X_{\mfk}\in\mfk$.

For $V\in T_o M$, let $(d\pi)^{-1}(V)$ denote (with abuse of notation) the unique $X\in\mfm$
such that $d\pi|_e X=V$. Then, one can define a metric on~$M$ by
\begin{equation}
\label{definition:metric-on-cosets}
\langle V, W\rangle_{go} := B\left( (d\pi)^{-1}(d\tau_g)^{-1} V, (d\pi)^{-1}(d\tau_g)^{-1}W\right)\ .
\end{equation}

 Homogeneous Riemannian manifolds constructed in the manner above were studied by Nomizu~\cite{nomizu} and are called \emph{normal} homogeneous Riemannian manifolds following M.~Berger's terminology~\cite{berger}. 

\subsection{The Levi-Civita connection on a normal homogeneous Riemannian manifold}
We would like to understand the Laplace-Beltrami operator on~$M$. To that end, we first need
to compute the Levi-Civita connection on~$M$. On Lie groups with bi-invariant Riemannian structure this computation (as well as the more involved curvature computation) was first carried out by Cartan~\cite{cartan}*{p. 64}. Then, it was extended by Nomizu~\cite{nomizu} to naturally reductive homogeneous Riemannian manifolds
(in particular, normal homogeneous spaces). We give here our own variation on this classical calculation.

\begin{lemma}
\label{lem:leibnitz}
Let $M$ be a normal homogeneous Riemannian manifold as above. Then,
for all $X, Y, Z\in \mfg$ one has
$$Z^+\langle X^+, Y^+ \rangle =
\langle [Z^+, X^+], Y^+ \rangle + \langle X^+, [Z^+, Y^+] \rangle$$
\end{lemma}
\remark. This formula does not make use of the $\Ad(G)$-invariance of~$B$ nor the positivity of~$B$ on~$\mfg$. It is valid
on any  homogeneous Riemannian manifold.
\begin{proof}
We confirm the validity of the formula only at the origin as this is all we use below.
\begin{equation*}
\begin{split}
 Z^+_{o}\langle X^+, Y^+ \rangle &=\frac{d}{dt}\Big|_{t=0} \langle (d\tau_{\exp tZ})^{-1} X^+_{(\exp tZ)o}, (d\tau_{\exp tZ})^{-1} Y^+_{(\exp tZ) o}\rangle_o
 \\
&\stackrel{\eqref{identity:tau-Ad-correspondence}}{=}\frac{d}{dt}\Big|_{t=0}\langle (\Ad \exp(-tZ)X)^+, (\Ad \exp(-tZ)Y)^+\rangle_o\\
&=-\langle [Z, X]^+, Y^+ \rangle_o - \langle X^+, [Z, Y]^+ \rangle_o\\
 &\stackrel{\eqref{form:X+Y+commutator}}{=}\langle [Z^+, X^+], Y^+ \rangle_o + \langle X^+, [Z^+, Y^+] \rangle_o
 \end{split}
\end{equation*}
\end{proof}
\begin{lemma}
\label{lem:koszul-reduced}
Let $M$ be a normal homogeneous Riemannian manifold as above.
Then, for all $X, Y, Z\in \mfg$ one has
$$2\langle\nabla_{X^+} Y^+, Z^+\rangle = \langle X^+, [Y^+, Z^+] \rangle - \langle Y^+, [Z^+, X^+] \rangle+
\langle Z^+, [X^+, Y^+]\rangle\ . $$
\end{lemma}
\remark. This formula is valid on any homogeneous Riemannian manifold.
\begin{proof}
By Koszul's formula  we have
\begin{equation*}
\begin{split}
  2\langle\nabla_{X^+} Y^+, Z^+\rangle&= X^+\langle Y^+, Z^+\rangle + Y^+\langle Z^+, X^+ \rangle -Z^+\langle X^+, Y^+ \rangle \\&
- \langle X^+, [Y^+, Z^+] \rangle + \langle Y^+, [Z^+, X^+] \rangle +\langle Z^+, [X^+, Y^+] \rangle\  .
\end{split}
\end{equation*}
We now substitute for the first three terms the expressions given by Lemma~\ref{lem:leibnitz}
and collect terms to obtain the claimed formula.
\end{proof}
\begin{lemma}
\label{lem:BonM-interpretation}
Let $M$ be a  normal homogeneous Riemannian manifold as above. Then, for all $X, Y, Z\in\mfg$
$$\langle X^+, [Y^+, Z^+] \rangle_o = -B(X_{\mfm}, [Y , Z]_\mfm)\ .$$
More generally,
$$\langle [X^+, Y^+], Z^+ \rangle_{go} = -B((\Ad(g^{-1})[X, Y])_\mfm, (\Ad(g^{-1})Z)_\mfm)\ .$$
\end{lemma}
\remark. This formula is valid on any homogeneous Riemannian manifold.
\begin{proof}
We only explain the formula at the origin, as this is all we  use below.
At the origin the formula follows immediately from~\eqref{form:X+Y+commutator}  and the definition of the metric on $M$ in~\eqref{definition:metric-on-cosets}.
\end{proof}

%
%

We are now able to compute the Riemannian connection:
\begin{proposition}[\cite{nomizu}*{Theorem 10.1}]
\label{prop:nat-reductive-geodesics}
Let $M$ be a normal homogeneous Riemannian manifold. Then,
$$ \forall X, Y, Z\in\mfm\quad \nabla_{X^+_o} Y^+ =\frac{1}{2}[X^+, Y^+]_o\ .$$
In particular, $$ \forall X\in\mfm\quad \nabla_{X^+_o} X^+ = 0\ .$$
\end{proposition}
\remark. The lemmas holds true on any naturally reductive homogeneous Riemannian manifold. 
 The curve $ge^{tX}o$ is a geodesic through $go$~\citelist{\cite{nomizu}\cite{oneill}*{ch. 11, prop. 25}}, but we will not use this fact.
 \begin{proof}
From Lemmas~\ref{lem:koszul-reduced} and~\ref{lem:BonM-interpretation} we know
$$2\langle \nabla_{X^+} Y^+, Z^+\rangle_o = -B(X, [Y, Z]_{\mfm})+B(Y, [Z, X]_{\mfm})-B(Z, [X, Y]_{\mfm})$$
A simple observation shows that condition~\eqref{cdn:ad-invariance} implies the
natural reductivity condition, namely,
\begin{equation}
    \label{cdn:naturally-reductive}
  \forall X, Y, Z\in\mfm\quad B([Z, X]_{\mfm}, Y)+B(X, [Z, Y]_{\mfm})=0\ .
\end{equation}
Thus, $$2\langle \nabla_{X^+} Y^+, Z^+\rangle_o = -B([X, Y]_{\mfm}, Z) = \langle [X^+, Y^+], Z^+ \rangle_o\ .$$
 \end{proof}

\subsection{The projected Casimir operator is $G$-invariant}

We denote by $U_{\mfg}$ the universal enveloping algebra of~$G$.
One then defines the Casimir element $\Omega\in U_\mfg$ corresponding to our fixed $\Ad(G)$-invariant positive definite form~$B$ as
$$\Omega=\sum_{j=1}^d  X_j^* X_j $$
where $(X_j)_{j=1}^d$ is any basis of $\mfg$ and $X_j^*$ is the dual basis with respect~$B$. One can verify that the element $\Omega$ is independent of the choice of basis~$X_j$ and, very importantly, that the $\Ad$-invariance of~$B$ implies~(see e.g.~\cite{bump-lie groups}*{Prop. 10.3} or~\cite{jacobson-lie algebras}*{Ch. III, \S7})
 that 
$$\Omega\in Z(U_{\mfg})\ ,$$
where $Z(U_{\mfg})$ denotes the center of $U_\mfg$.


Consider the left invariant differential operator on $G$ corresponding to the Casimir element given by
 $$\tilde{\Omega}:=\sum_{j=1}^{n}  \tilde{X}^*_j\tX_j$$
 
 \begin{proposition}
    \label{prop:casimir-bi-invariant}
   The Casimir operator $\tilde{\Omega}$ is  both left and right invariant (i.e.
   it commutes with left and right translations). In other words,
   $$\bar{\Omega}=\tilde{\Omega}\ ,$$
    where $\bar{\Omega}=\sum_{j=1}^{n} \bar{X}^*_j \bar{X}_j$.
 \end{proposition}
\begin{proof}
A left invariant differential operator on a connected Lie group~$G$ is right invariant if and only if it commutes with all
left invariant vector fields $\tX$ on $G$ (see~\cite{helgason-book2}*{Ch. II, Lemma 4.4}).
Since the Casimir element is in the center of $U_\mfg$
it follows that $\tilde{\Omega}$ is right invariant.
\end{proof}

The next theorem makes use of the full structure of the normal homogeneous Riemannian  manifold $M$. In particular, it is not enough to assume that $M$ is naturally reductive.

\begin{theorem}
\label{thm:casimir-projection}
Let $M$ be a normal homogeneous Riemannian manifold, and let $(X_j)_{j=1}^{n}$ be a basis for~$\mfg$ as above. Then, the projected Casimir operator on~$M$,
$$ \Omega^+:=\sum_{j=1}^{n} (X^*_j)^+ X^+_j\ ,$$
is $G$-invariant. Moreover, $X^+ \Omega^+ = \Omega^+ X^+$ for all $X\in\mfm$.
\end{theorem}
\remark. In this theorem, we could take the Casimir operator corresponding to any bi-invariant Riemannian metric on~$G$, not necessarily
the one corresponding to our fixed background Riemannian metric on~$M$.
\begin{proof}
Observe that  
$$ (\Omega^+ f)\circ\pi  = \bar{\Omega} (f\circ\pi) =\tilde{\Omega} (f\circ\pi) \ .$$
Hence, 
\begin{equation*}
\begin{split}
(\Omega^+ f)\circ \tau_g \circ\pi &= (\Omega^+f)\circ\pi\circ L_g = \tilde{\Omega} (f\circ\pi)\circ L_g\\&=\tilde{\Omega}(f\circ\pi\circ L_g)=\tilde{\Omega}(f\circ\tau_g\circ\pi)=\Omega^+(f\circ\tau_g)\circ\pi\ ,
\end{split}
\end{equation*}
implying that $(\Omega^+f)\circ\tau_g=\Omega^+(f\circ\tau_g)$.
Also, for $X\in \mfm$
$$ (X^+ \Omega^+ f)\circ\pi =\bar{X}\bar{\Omega} (f\circ\pi) = \bar{X}{\tilde{\Omega} (f\circ\pi})=\tilde{\Omega}\bar{X}(f\circ\pi)=\bar{\Omega}\bar{X} (f\circ\pi)
=(\Omega^+ X^{+} f)\circ\pi\ .$$
\end{proof}
\subsection{The Laplace-Beltrami operator vs. the Casimir operator}
In this section we identify, roughly speaking, the Laplace-Beltrami operator with the Casimir operator.
In the context of \emph{rank one} symmetric spaces this statement can be found in~\cite{faraut}*{p. 378},
and in the context of Riemannian symmetric spaces of the \emph{noncompact} type it can be found in~\cite{helgason-book2}*{Ch. II, exercise A.4}. 
The argument for Riemannian symmetric spaces of the \emph{compact} type (or normal homogeneous \emph{compact} Riemannian manifolds) is essentially the same and we  include it here for completeness.
\begin{theorem}
\label{thm:laplace-sum-of-squares}
Let $M$ be a normal homogeneous compact Riemannian manifold as above. Let $\Delta_M$ be the Laplace-Beltrami
operator on~$M$ and $\Omega^+$ be the projected Casimir operator as in Theorem~\ref{thm:casimir-projection}.
Then, 
\begin{equation} 
\label{eqn:lap=cas}
  \Delta_M =\Omega^+\ .
\end{equation}
In particular, the operator $\Delta_M$ can be written as a sum of squares of vector fields on $M$ which commute with $\Delta_M$.
\end{theorem}
\begin{proof}
Let $X_1, \ldots X_m, X_{m+1}, \ldots X_n$ be an orthonormal basis of $\mfg$ with respect to the form $B$ which is  compatible with
the decomposition $\mfg=\mfm\oplus\mfk$, i.e. $X_j\in \mfm$ for $1\leq j\leq m$
and $X_j\in\mfk$ for $m+1\leq j\leq n$.
Observe that the projected Casimir operator~$\Omega^+$ in Theorem~\ref{thm:casimir-projection}  takes with this choice of basis the form
$$ \Omega^+=\sum_{j=1}^m (X_j^+)^2\ .$$
Since both $\Delta_M$ and $\Omega^+$ are $G$-invariant by Theorem~\ref{thm:casimir-projection}, then it
suffices to verify they coincide  at the origin.
By definition
$$\Delta_M f = \tr (Y\mapsto \nabla_Y \grad f)\ .$$ 
Since the basis $(X_j^+)_{j=1}^m$ is orthonormal at the origin we have that
\begin{equation*}
\begin{split}
  (\Delta_M f)(o) &= \sum_{j=1}^m \langle \nabla_{X_j^+} \grad f, X_j^+ \rangle_o =
\sum_{j=1}^m (X_j^+)^2 f(o) - \sum_{j=1}^m (\nabla_{X_j^+} X_j^+)_o f\\
&\stackrel{\textrm{Prop.}~\ref{prop:nat-reductive-geodesics}}{=}\sum_{j=1}^m (X_j^+)^2 f(o) =\sum_{j=1}^n (X_j^+)^2 f(o) =\Omega^+_o f\ .
\end{split}
\end{equation*}
In applying Proposition~\ref{prop:nat-reductive-geodesics} we are making use of the fact that $X_j\in\mfm$ for $1\leq j\leq m$ and then we make use of the fact that  $X_j^+|_o=0$ for $m+1\leq j\leq n$.
\end{proof}

\section{Powers of the Laplace and harmonic functions: Proof of Theorem~\ref{thm:homog}}
The proof of Theorem~\ref{thm:homog} is now immediate.
The exact same argument which lets us deduce Theorem~\ref{thm:spheres} from Proposition~\ref{prop:spheres-sum-of-squares}, permits us also to  deduce the absolute monotonicity result in Theorem~\ref{thm:homog} from Theorem~\ref{thm:laplace-sum-of-squares}.

\section{Discussion: Symmetric spaces of the noncompact type}
\label{sec:discussion}
In this note we have not addressed the highly important case of a Riemannian symmetric space $M=G/K$ of the \emph{noncompact} type and, more generally, the case of a noncompact, semisimple Lie group $G$ with $K$ a compact Lie subgroup and where the Riemannian metric on~$M$ is induced from a bi-invariant pseudo-Riemannian mentric on~$G$.
The arguments in the preceding sections do show that one can write the Laplace-Beltrami operator~$\Delta_M$ on $M$ as a sum $\sum_j \pm (X_j^+)^2$, where the vector fields $X_j^+$ commute with $\Delta_M$. However, it is far from clear what implications this expression has for harmonic functions.
It seems to be highly interesting to understand positivity and convexity properties for the sequence of quadratic forms on harmonic functions coming from $\Delta^k\lvert h\rvert^2$.

\end{document}